\documentclass[a4paper,draft,intlimits,reqno,10pt,oneside]{amsart}
\usepackage{amssymb,enumerate,eucal,verbatim}

\usepackage{amsmath}
\usepackage{amssymb}
\usepackage{amsthm}
\usepackage{latexsym}

\newtheorem{lemma}{Lemma}[section]
\newtheorem{theorem}[lemma]{Theorem}
\newtheorem{proposition}[lemma]{Proposition}
\newtheorem{corollary}[lemma]{Corollary}

\newtheorem{definition}[lemma]{Definition}
\newtheorem{fact}[lemma]{Fact}

\newtheorem{problem}[lemma]{Problem}

\theoremstyle{definition}

\def\({\bigl(}
\def\){\bigr)}

\def\N{I\!\!N}

\def\fr#1#2{\dfrac{\lower 1 pt\hbox{\mathstrut}#1}{\mathstrut #2}}

\def\N{I\!\!N}

\def\({\big(}
\def\){\big)}

\def\N{I\!\!N}

\def\B{\mathbf{B}}
\def\S{\mathbf{S}}

\def\EndDef{\end{definition}}
\def\Def{\begin{definition}}
\def\EndFact{\end{fact}}
\def\Fact{\begin{fact}}
\def\Pf{\begin{proof}}
\def\EndPf{\end{proof}}

\def\Lm{\begin{lemma}}
\def\EndLm{\end{lemma}}

\def\EndCor{\end{corollary}}
\def\Cor{\begin{corollary}}

\def\Thmn{\begin{theorem}}
\def\EndThm{\end{theorem}}
\def\Propn{\begin{proposition}}
\def\EndProp{\end{proposition}}

\newcommand\conv{\mathrm{conv}}

\renewcommand{\mod}{{/}}

\newcommand{\scp}{\mathrm{SCP}}
\newcommand{\gds}{\mathrm{GDS}}

\def\N{I\!\!N}

\begin{document}

\title{Note on Kadets Klee property and Asplund spaces}

\author{Petr H\'ajek}
\address{Mathematical Institute\\Czech Academy of Science\\\v Zitn\'a 25\\115 67 Praha 1\\Czech Republic\\
Department of Mathematics} 
\address{Faculty of Electrical Engineering\\ Czech Technical University in Prague\\ Zikova 4, 166 27\\ 
Prague, Czech Republic}
\email{hajek@math.cas.cz}
\thanks{The first named author was financially supported by a grant GACR 201/11/0345.}

\author{Jarno Talponen}
\address{Aalto University, Institute of Mathematics, P.O. Box 11100, FI-00076 Aalto, Finland}
\email{talponen@iki.fi}

\date{\today}

\thanks{The first author was supported in part by Institutional Research Plan AV0Z10190503 and GA\v CR P201/11/0345.
This paper was prepared as the second author enjoyed the warm hospitality of the Czech Academy of Sciences in Autumn
2011. The visit and research was supported in part by the V\"ais\"al\"a foundation.}

\begin{abstract}
A typical result in this note is that if $X$ is a Banach space which is a weak Asplund space and has the $\omega^*$-$\omega$-Kadets Klee property, then $X$ is already an Asplund space.
\end{abstract}

\maketitle

\section{Introduction}

Recall that a Banach space $(X,\|\cdot\|)$ (resp. its norm $\|\cdot\|$) is said to be locally
uniformly rotund (LUR for short) if for each $x,x_n\in X,\ n\in \N,$ with $\|x\|=\|x_n \|=1$ 
and $\lim_{n\to \infty}\|x_n + x\|=2$ it follows that $\lim_{n\to \infty}\|x_n - x\|=0$.

We say that a Banach space $(X,\|\cdot\|)$ has the Kadec property
(resp. Kadec-Klee property, KK for short)
if the norm and weak topologies
coincide on the unit sphere $S_X$ (resp. weakly convergent sequences from $S_X$ are norm convergent).

In \cite{tro3} it is shown that if $\ell_1\not\hookrightarrow X$ and $\|\cdot\|$ is KK then it 
is also Kadec. The Kadec property has been studied extensively in the literature,
and we refer to \cite{dgz} and \cite{motv} for the background. Let us mention only few of
the main results in this area which relate the Kadec and LUR properties.
It is easy to observe that if $\|\cdot\|$ is LUR then $\|\cdot\|$ has the Kadec property.
Troyanski proved \cite{tro1}, \cite{tro2} that $X$ has an equivalent LUR renorming if and only if
it has an equivalent Kadec and rotund renormings. Haydon \cite{hay}, \cite{haydon-trees} gave examples showing that
the rotundity assumption is essential, constructing trees such that $C(T)$ has KK but no rotund renorming. 
In \cite{motv} the authors prove that if $X$ has the RNP property (in particular if $X$ is a dual of some Asplund space) then $X$ has an LUR renorming if and only if it has a Kadec renorming.
In the present note we are interested in the dual situation.

\begin{definition}
Let $(X,\|\cdot\|)$ be a Banach space. We say that $(X^*,\|\cdot\|^*)$ has the

1. Dual Kadec property (K$^*$) if the $w^*$ and norm topologies coincide on the dual unit sphere
$S_{X^*}$.

2. Dual  weak Kadec property (wK$^*$) if the $w^*$ and weak topologies coincide on the dual
unit sphere $S_{X^*}$.

3. Dual  Kadec-Klee property (KK$^*$) if the $w^*$-convergent and norm convergent
sequences coincide on the dual unit sphere
$S_{X^*}$.

3. Dual  weak Kadec-Klee property (wKK$^*$) if the $w^*$-convergent and weak convergent
sequences coincide on the dual unit sphere
$S_{X^*}$.
\end{definition}

These notions are closely related to the Asplund property of $X$. In particular,
Namioka and Phelps \cite{np} showed that if $X^*$ has the K$^*$ then $X$ is an Asplund space.
Raja strengthened this by showing that $X^*$ has a dual LUR renorming (which implies
automatically the K$^*$ property and also that $X$
is an Asplund space) if and only if it has a K$^*$ renorming if and only if 
it has a wK$^*$ renorming \cite{r1}, \cite{r2}.

It is therefore natural to ask whether the formally weaker sequential versions of these properties
lead to some positive result. The problem whether the existence of a KK$^*$ renorming of $X^*$ implies that $X$ 
is an Asplund space has apparently been posed by Godefroy, and it seems to be open.
In our note we are going to address the still weaker property wKK$^*$ of $X^*$.
It follows that for every renorming of a given Grothendieck space $X$, its dual $X^*$
has the wKK$^*$. Hence the property may not imply that $X$ is Asplund, in general (e.g. $\ell_\infty$).

\subsection{Preliminaries}

We denote by $X$, $Y$, $Z$ and $E$ real Banach spaces. The closed unit ball and the unit sphere of $X$ are denoted by
$B_{X}$ and $S_{X}$, respectively. For suitable background information on general Banach space theory and notations we refer to \cite{fhhmz}.

We say that a subspace $Y\subset X$ is \emph{coseparable} in $X$ if $X\mod Y$ is separable (see \cite{talponen} for discussion). Following \cite{fabian_book} we call $X$ a Gateaux differentiability space ($\gds$), if each continuous convex function on $X$ is Gateaux differentiable in a dense set. If the sets in the previous definition are additionally
$G_{\delta}$-sets, then $X$ is a \emph{weak Asplund space}. A Banach space $X$ has the \emph{Grothendieck property} if in the dual $X^*$ the $w^*$-convergent and $w$-convergent sequences coincide. Recall that a Banach space has the \emph{$1$-separable complementation property} ($1$-$\scp$), if for each separable subspace $Y\subset X$ there exists a separable subspace $Z\subset X$ such that $Y\subset Z$ and $Z$ is contractively complemented in $X$. If $A\subset X$ is a subset, then we denote by $[A]$ the closed linear span of $A$. The duality mapping $J\colon \S_X \to 2^{\S_{X^*}}$ is given by $J(x)=\{x^* \in \S_{X^*}\colon x^* (x)=1\}$ for $x\in \S_X$.

\section{Results}

\begin{theorem}\label{thm: main}
Let $X$ be a Banach space which can be renormed in such a way that $X^*$ satisfies both the wKK$^*$ and one of the following conditions:
\begin{enumerate}
\item[$(i)$]{$J\colon \S_{X}\to 2^{\S_{X^*}}$ is $w^*$-first-countably valued.}
\item[$(ii)$]{The coseparable subspaces of $X$ are preserved in countable intersections.}
\item[$(iii)$]{$X$ has $1$-$\scp$.}
\item[$(iv)$]{$\B_{X^*}$ is $w^*$-countably compact (e.g. $X$ is a Gateaux differentiability space or weak Asplund)}.
\end{enumerate}
Then $X$ is an Asplund space.
\end{theorem}

Observe that wKK$^*$ property of $X^*$ is clearly weaker than KK$^*$ property and the Grothendieck property. Since there exists non-Asplund Grothendieck spaces, e.g. $L^{\infty}$, we conclude that the assumptions $(i)$-$(iv)$ cannot be plainly removed in the theorem. Shortly we will discuss the conditions $(i)$ and $(ii)$. The first condition is of course much weaker than imposing $J$ to be single-valued, i.e. $X$ to be Gateaux smooth. Recall that WLD, Plichko and certain other classes of Banach spaces have $1$-$\scp$, see \cite[p.105]{bos}.

\begin{proposition}\label{prop}
Let $x^* \in \S_{X^*},\ x\in \S_{X}$ be such that $x^*(x)=1$. Let $(x_{n}^* )\subset \B_{X^* }$ be a sequence such that
$x^* \in \overline{(x_{n}^*)}^{w^*}$. Assume that there exists a sequence $(U_{n})$ of relatively $w^*$-open
subsets of $J(x)$ such that
\[\bigcup_{n}U_{n} = \bigcup_{n}\overline{U_{n}}^{w^*} = J(x)\setminus \{x^*\}.\]
Then there exists a subsequence $(n_{k})\subset \N$ such that $x_{n_{k}}^* \stackrel{w^*}{\longrightarrow}x^*$ as $k\to \infty$.
\end{proposition}

First we will comment on the assumptions regarding $J(x)$ in the $w^*$-topology, which is subsequently termed as $K$.
If $J(x)$ is norm-separable, then it is in particular hereditarily Lindel\"of in the $w^*$-topology.
The same is true if $K$ is metrizable. Recall that a compact hereditarily Lindel\"of space is first-countable
(see e.g. \cite[p.473]{bg}).

The following condition is also useful: there exists a sequence $(x_{n})\subset X\subset X^{**}$ such that
the sequence $y_{n}=x_{n}|_{[J(x)]}\in [J(x)]^*$ satisfies
\[\overline{[(y_{n})]}^{w^*}=[J(x)]^*,\]
that is, the sequence $(x_{n})$ separates $[J(x)]$. This condition, or the first-countability of $K$
imply, in particular, that there exists a sequence $(x_{n})\subset X$ such that for any $y^* \in K,\ y^* \neq x^*,$ there
is $n$ such that $(y^* -  x^*)(x_{n})\neq 0$. Observe that this condition in turn implies the topological assumption
of Proposition \ref{prop}.

\begin{proof}[Proof of Proposition \ref{prop}]
Fix a sequence $(U_{n})$ of $w^*$-open subsets of $K$ such that
\[\bigcup_{n} U_{i}=\bigcup_{n} \overline{U_{i}}^{w^*} = K\setminus \{x^*\}.\]
Let $(\tilde{U}_{n})$ be a sequence of $w^*$-open sets of $\B_{X^*}$ such that
$U_n \subset \tilde{U}_{n}\cap J(x)$ and $x^* \notin \overline{\tilde{U}_{n}}^{w^*}$ for $n\in\N$.
Indeed, $K$ is compact Hausdorff space, thus normal, hence completely regular, and so we may separate
$\overline{U_{n}}^{w^*}$ and $x^*$ by a continuous function $g\colon K\to [0,1]$. Then a Tietze extension
$\tilde{g}\colon \B_{X^*}\to [0,1]$ of $g$ yields a suitable $w^*$-open set $\tilde{U}_{n}$. Let $(V_{i})$ be a sequence defined by $V_{2i}=\tilde{U}_{i}$ and $V_{2i-1}=\{y^*\in \B_{X^*}:\ x(y^*)<1-1/i\}$.

We define a subsequence of $(x_{n}^*)$ by a diagonal argument as follows. First, we remove from the sequence
the functionals included in $V_{1}$. From the resulting sequence, say $s_{1}$, we remove the functionals included in
$V_2$, starting from the index third in order. Next, from the previously obtained sequence $s_{2}$ we pick a new sequence $s_3$ by removing the functionals included in $V_3$ starting from the fourth index. We proceed in this manner.
Observe that $s_{n}$ includes the $n$ first elements of $s_{n-1}$. Therefore $s_{n}$ converges coordinate-wise to a sequence $s$, which is a subsequence of $(x_{n}^*)$. Here $(x_{n})\supset s_{1}\supset s_{2}\supset \ldots \supset s$.
Let us denote this resulting sequence $s$ by $(x_{n_{k}}^*)$. By the construction $x_{n_{k}}^* (x)\to 1$ as $k\to \infty$.

Assume to the contrary that $x_{n_{k}}^*$ does not $w^*$-converge to $x^*$.
Thus there is a point $y\in X$ and a further subsequence $(n_{k_{j}})_{j}$ such that
$(x^* - x_{n_{k_{j}}}^* )(y)\to c\neq 0$ as $j\to \infty$. By the $w^*$-compactness of $\B_{X^*}$ we get that
there is $z^* \in J(x),\ z^* \neq x^* ,$ such that $z^*$ is a $w^*$-cluster point of $(x_{n_{k_{j}}}^*)_j$.
Then, according to the assumptions there is $i_{0}\in \N$ such that $z^* \in \tilde{U}_{i_{0}}$. By reviewing the construction of $s_{2i_{0}}$ we observe the impossibility of $z^*$ being a $w^*$-cluster point of $s_{2i_{0}}$. This contradiction finishes the proof.
\end{proof}

\begin{proposition}
Let $X$ be a Banach space which can be renormed to be both $\gds$ and $1$-$\scp$.
Then the coseparable subspaces of $X$ are preserved in taking countable intersections.
\end{proposition}
\begin{proof}
Let $X$ satisfy the assumption about the renorming. Since the conclusion of the result obviously does not depend on equivalent renormings, we may restrict our attention to a given $1$-$\scp$ Gateaux differentiability space $X$.
Let $(Z_{n})$ be a sequence of coseparable subspaces of $X$. We denote by $q_{n}\colon X \to X\mod Z_n$
the canonical quotient mappings.

By applying the fact that $X\mod Z_{n}$ are separable and the $\gds$ property
we may choose subsets $\{z^{(n)}_{k}\in \S_{X}:\ n,k\in \N\}$ and $\{f^{(n)}_{k}\in \S_{X^*}:\ n,k\in \N\}$ such that
\begin{enumerate}
\item[(a)]{Each $z^{(n)}_{k}$ is a smooth point.}
\item[(b)]{$\S_{X\mod Z_{n}}\subset \overline{q_{n}(\{z^{(n)}_{k}:\ k\in \N\})}$ for each $n$.}
\item[(c)]{Each $f^{(n)}_{k}$ attains its norm at $z^{(n)}_{k}$.}
\item[(d)]{$\{f^{(n)}_{k}:\ k\in \N\}$ $1$-norms $[z^{(n)}_{k}:\ k\in \N]$ for each $n$.}
\end{enumerate}
Indeed, the last condition is obtained by applying a back-and-forth recursion of countable length.
From $(b)$ and $(d)$ we obtain that $\B_{Z_{n}^\bot}\subset \overline{\conv}^{w^{\ast}}(f_{k}^{(n)}:\ k\in \N)$ for $n$.
This means that
\[\bigcap_{k}\ker f^{(n)}_{k} \subset Z_{n}\quad \text{for}\ n\in\N.\]

According to $1$-$\scp$ there exists a separable subspace $E\subset X$ with $\{z^{(n)}_{k}:\ n,k\in \N\}\subset E$ and
a norm-$1$ linear projection $P\colon X\to E$.

According to the Gateaux smoothness of the points $z^{(n)}_{k}$ we have that
\[f^{(n)}_{k}\circ P = f^{(n)}_{k}\quad \mathrm{for}\ n,k\in \N,\]
since the functionals are norm attaining at the corresponding points.
Thus $\ker P\subset \bigcap_{n,k}\ker f^{(n)}_{k}$ where $\ker P$ is
clearly a coseparable subspace. On the other hand, $\bigcap_{k}\ker f^{(n)}_{k}\subset Z_{n}$ for $n$, so that
$\bigcap_{n} Z_{n}$ is necessarily a coseparable subspace of $X$.

\end{proof}

\begin{proof}[Proof of Theorem \ref{thm: main}]
Let $Y\subset X$ be a separable subspace. Our aim is to show that $Y^{*}$ is separable.

Observe that $\B_{Y^{\ast}}$ is separable and metrizable in the $w^{*}$-topology.
It is easy to see that $\S_{Y^*}$ is such as well. Let us pick a $w^*$-dense set
$\{y_{n}^* :\ n\in \N\}\subset \S_{Y^*}$. Our aim is to show that
\[\overline{\conv}(\{y_{n}^* :\ n\in \N\})=\B_{Y^*}.\]
Towards this, fix a norm-attaining functional $y^* \in \S_{Y^*}$. Let $y\in \S_{Y}$ be such that $y^* (y)=1$.

By using $w^{\ast}$-compactness and $w^*$-metrizability of the dual unit ball we obtain that there exists
a subsequence $(y^{*}_{n_{k}})_{k\in \N}\subset \{y_{n}^* :\ n\in \N\}$ which $w^*$-converges to $y^*$,
since the countable set was chosen to be $w^*$-dense.

Our aim is to find a further subsequence of $(n_{k})$, indexed by $m$,
and a sequence $(z^{*}_{m})_{m\in \N}\subset \S_{X^*}$ and $z^* \in \S_{X^{\ast}}$, which are Hahn-Banach extensions of
$(y^{*}_{n_{k_{m}}})_{m\in \N}$ and $y^*$, respectively, such that
$z^{*}_{m}\stackrel{w^*}{\longrightarrow} z^*$ in $X^*$ as $m\to \infty$.
Namely, in such a case we obtain by the assumptions that $z^{*}_{m}\stackrel{w}{\longrightarrow} z^*$
and then Mazur's theorem yields that $z^* \in \overline{\conv}(z^{*}_{m}:\ m\in \N)$.
This means that $y^* \in \overline{\conv}(y^{*}_{n_{k_{m}}}:\ m\in \N)$, and since $y^*$ was an arbitrary norm-attaining
functional, we obtain by the Bishop-Phelps theorem that $\overline{\conv}(y^{*}_{n}:\ n\in \N)=\B_{Y^*}$.

Fix Hahn-Banach extensions $(x^*_{k})$ and $x^*= x_{0}^*$ of $(y^*_{n_{k}})$ and $y^*$, respectively.
Observe that $x^*_k (y)\to 1=x^*_{0}(y)$ as $k\to \infty$.
To finnish the proof by finding the appropriate subsequence $(n_{k_{m}})_{m}$ and Hahn-Banach extensions, we
proceed in parallel steps regarding the assumptions $(i)$-$(iv)$.

\noindent {\it Assumption $(i)$:} The set $J(y)$ is $w^*$-first-countable. Proposition \ref{prop} and the discussion following yield that there is a subsequence $(k_{m})$ and $z^* \in J(y)\subset \S_{X^*}$ such that
$x^*_{k_{m}} \stackrel{w^*}{\longrightarrow} z^*$ as $m\to \infty$.
Note that $z^*$ is a Hahn-Banach extension of $y^*$, since by the selection of $(x^*_{k_{m}})$ it holds that
$x^*_{k_{m}}|_{Y}\stackrel{w^*}{\longrightarrow}y^*$ in $Y^{\ast}$ as $m\to\infty$.
We write $z^*_{m}=x^*_{k_{m}}$ to get the required extensions.

\noindent {\it Assumption $(ii)$:} The coseparable subspaces of $X$ are preserved in countable intersections.
Observe that $Z=\bigcap_{k\geq 0}\ker x^*_{k}\subset X$ is coseparable. Thus we may consider
$\{x^*_{k}:\ k\geq 0\}\subset \B_{(X\mod Z)^{\ast}}=\B_{Z^{\bot}}$, where the unit ball is $w^*$-compact and $w^*$-metrizable. Thus we may extract a $w^*$-converging subsequence $(x^*_{k_{m}})$. Let $z^*=w^*$-$\lim_{m\to \infty}x^*_{k_{m}}$. We put $z^*_{m}=x^*_{k_{m}}$ for $m$. By the selection of the sequence $(n_{k_{m}})$ it is again clear that $z^*\in \S_{X^*}$ is a Hahn-Banach extension of $y^* \in \S_{Y^*}$.

\noindent {\it Assumption $(iii)$:} $X$ has $1$-$\scp$. Let $P\colon X\to E$ be a norm-$1$ projection where $E\subset X$ is a separable subspace containing $Y$. Similarly as in the previous step, we observe that
$\{x^*_{k}\circ P:\ k\geq 0\}\subset \B_{(X\mod \ker P)^{\ast}}$ contains a subsequence $(z^*_{m})=(x^*_{k_{m}}\circ P)$ and
$z^*=z^* \circ P \in \S_{(X\mod \ker P)^{\ast}},\ z^*(y)=1,$ such that $x^*_{k_{m}}\circ P \stackrel{w^*}{\longrightarrow} z^*$ in $(X\mod \ker P)^*$ as $m\to \infty$. Since $Y\subset E$, we see similarly as above that $(x^*_{k_{m}}\circ P)|_{Y}=y^*_{n_{k_{m}}}$ for $m\in\N$. This yields that $z^* |_{Y}=y^*$.

\noindent {\it Assumption $(iv)$:} The dual unit ball $\B_{X^*}$ is $w^*$-sequentially compact.
This is the case, for instance, if $X$ is GDS (see \cite[Thm. 2.1.2]{fabian_book}).
Thus there is a subsequence $(k_{m})\subset \N$ and $z^*\in \S_{X^*}$ such that $z^*(y)=1$ and $w^*$-$\lim_{m\to\infty}x^*_{k_m}=z^*$. Clearly $z^*|_{Y}=y^*$ and
we put $z^*_{m}=x^*_{k_m}$.
\end{proof}

Finally, we will apply Theorem \ref{thm: main} to give a new, rather simple proof for the Josefson-Nissenzweig Theorem
in a subclass of Banach spaces.

\begin{proposition}
Let $X$ be an infinite-dimensional Banach space which is not an Asplund space and satisfies one of the conditions $(i)$-$(iv)$ in Theorem \ref{thm: main}. Then there is a sequence $(z_{k}^*)\subset \S_{X^*}$ such that $w^*$-$\lim_{k\to\infty}z_{k}^* =0$ and $0\notin \overline{\conv}(z_{k}^*)$.
\end{proposition}
\begin{proof}
Let $Y\subset X$ be a separable subspace with non-separable dual.
Thus $Y$ is not Asplund but $\B_{Y^{\ast}}$ is $w^*$-metrizable. By using Theorem \ref{thm: main} and its proof we conclude that there exists
a sequence $(y_{n}^*)\subset \S_{Y^*}$ and $y^* \in \S_{Y^*}$ such that
$y_{n}^* \stackrel{w^*}{\longrightarrow} y^*$ as $n\to \infty$ but $y^*\notin \overline{\conv}(y_{n}^* )$.
By proceeding similarly as in the proof of Theorem \ref{thm: main} we may find a subsequence $(n_k)$ and
Hahn-Banach extensions $x_{n_k}^*$ and $x^*$ of $y_{n_k}^*$ and $y^*$, respectively, such that
$x_{n_k}^* \stackrel{w^*}{\longrightarrow} x^*$ in $X^*$ as $k\to \infty$.
Define $u_{k}^*=x_{n_k}^* - x^*$ for $k$. This is a $w^*$-null sequence.
Observe that $0\notin \overline{\conv}(u_{k}^*)$ since $y^*\notin \overline{\conv}(y_{n}^* )$. By the geometric Hahn-Banach theorem there is a separating functional $f\in X^{**}$ such that, say, $f(u_{k}^*)>1$ for $k$. We define the required sequence by $z_{k}^* =u_{k}^* / \|u_k^* \|$ where $\|u_k^*\|\leq 2$ for all $k\in \N$. Observe that $f(z_k^*)> 1/2$ for $k\in N$, so that $0\notin \overline{\conv}(z_{k}^*)$.
\end{proof}

\subsection{Final remarks}

Regarding the abovementioned problem of Godefroy, the example \cite{fe} of countably tight compact $K$ without any convergent sequence provides a possible direction for searching a counterexample $C(K)$ space. 

\begin{problem}
Find examples with KK failing K, resp. wKK$^*$ failing K$^*$.
Consider the long James space $J(\Gamma)$. Does it have the K$^*$?
\end{problem}

Another related subject is the following weakenning of the well-known problem whether
every dual to a separable space without a copy of $\ell_1$ is LUR renormable (see e.g. \cite{hmo}).

\begin{problem}
Let $X$ be separable $\ell_1\not\hookrightarrow X$. Does $X^*$ admit a KK renorming?
\end{problem}


\begin{thebibliography}{DGZ}

\bibitem[BG]{bg} M. BELL, J. GINSBURG, \emph{Uncountable discrete sets in extensions and metrizability},
Canad. Math. Bull, 25, (1982), 472-477.

\bibitem[BoP]{bop} P. Borodulin and G. Plebanek, {\em On sequential properties of Banach spaces, 
spaces of measures and densities}, Czech Math. J. 60, (2010), 381-399.

\bibitem[DGZ]{dgz} R. Deville, G. Godefroy, and V. Zizler, {\it Smoothness and
renormings in Banach spaces,} Pitman Monographs and Surveys in Pure and
Applied Mathematics, 64, 1993.\par

\bibitem[Dieu]{dieudonne} J. Dieudonn\' e, {\em Foundations of Modern Analysis}, Academic Press, 1960.

\bibitem[F]{fabian_book} M. Fabian, {\em Gateaux differentiability of convex functions and topology: 
weak Asplund spaces}, Canadian Mathematical society series of Monographs and adcanced texts, 
Wiley-Interscience publication, 1997.

\bibitem[FHHMZ]{fhhmz} M. Fabian, P. Habala, P. H\'ajek, V. Montesinos, V. Zizler,
{\it Banach Space Theory, The Basis for Linear and Nonlinear Analysis}, CMS Books in Mathematics, Springer 2011.\par

\bibitem[Fe]{fe} V.V. Fedorcuk, {\em A compact space having the cardinality of the continuum with no convergent sequences},
Math. Proc. Cambridge Phil. Soc. 81 (1977), 177--181.

\bibitem[HMVZ]{bos} P. H\'ajek, V. Montesinos, J. Vanderwerff, and V. Zizler,
{\em Biorthogonal systems in Banach spaces}, CMS Books in Mathematics, Canadian
Mathematical Society, Springer Verlag, 2007.

\bibitem[Hay]{hay} R. Haydon, {\em A counterexample to several questions about scattered compact
spaces}, Bull. LMS 22 (1990), 261--268.


\bibitem[Hayd3]{haydon-trees} R. Haydon, {\it Trees in renorming theory},
Proc.\ London Math.\ Soc.~{\bf 78} (1999), 541--584.

\bibitem[HMO]{hmo} R. Haydon, A. Molto and J. Orihuela, {\em Spaces of functions with 
countably many discontinuities}, Isr. J. Math. 158 (2007), 19--39.

\bibitem[Me]{me} SA. Mercourakis, {\em Some remarks on countably determined measure and uniform distribution 
of sequences}, Monatsh. Math. 121 (1996), 79--111.

\bibitem[MOTV00]{motv00} A. Molto, J. Orihuela, S. Troyanski and M. Valdivia,
{\em Kadec and Krein-Milman properties},
C.R. Acad. Sci. 331 (2000), 459--464.

\bibitem[MOTV]{motv} A. Molto, J. Orihuela, S. Troyanski and M. Valdivia, {\em A Nonlinear Transfer Technique
for Renorming}, Springer LNM 1951 (2009).

\bibitem[NP]{np} I. Namioka and R.R. Phelps, {\em Banach spaces which are Asplund spaces}, Duke Math. J.
42 (1975), 735--750.

\bibitem[R2]{r2} M. Raja, {\em On dual locally uniformly rotund norms}, Isr. J. Math. 129 (2002), 77-91.


\bibitem[R1]{r1} M. Raja, {\em Locally uniformly rotund norms}, Mathematika 46 (1999), 343--358.


\bibitem[T]{talponen} J. Talponen, Extracting long basic sequences from systems of dispersed vectors,
Adv. Math. 231 (2012), 1068--1084.

\bibitem[Tro1]{tro1} S.L. Troyanski, {\em On equivalent locally uniformly convex norms},
C.R. Bulg. Acad. Sci. 32 (1979), 1167--1169.

\bibitem[Tro2]{tro2} S.L. Troyanski, {\em Construction of equivalent norms for certain
local characteristics  with rotundity and smoothness by means of martingales},
Proc. of the 14th Spring Conference of the Union of Bulgarian Mathematicians (1985).

\bibitem[Tro3]{tro3} S.L. Troyanski, {\em On a Property of the Norm which is Close to Local Uniform
Rotundity}, Math. Ann. 271 (1985), 305--313.

\end{thebibliography}
\end{document}